\documentclass[a4paper,12pt]{amsart}
\usepackage[utf8]{inputenc}
\usepackage{a4wide}
\usepackage{amsfonts,amssymb,amsthm,amsmath}
\usepackage{color,url,enumitem,hyperref}

\theoremstyle{plain}
\newtheorem{theorem}{Theorem}[section]

\newtheorem{lemma}[theorem]{Lemma}
\newtheorem{corollary}[theorem]{Corollary}
\newtheorem{proposition}[theorem]{Proposition}
\newtheorem{conjecture}[theorem]{Conjecture}

\newtheorem*{problemintro*}{Problem}

\theoremstyle{definition}

\newtheorem{observation}[theorem]{Observation}
\newtheorem{problem}[theorem]{Problem}

\definecolor{darkblue}{rgb}{0,0,0.7} 

\DeclareMathOperator{\susp}{\Sigma} 
\DeclareMathOperator{\lk}{lk} 
\DeclareMathOperator{\st}{st} 
\DeclareMathOperator{\antist}{ast}
\DeclareMathOperator{\bd}{\partial} 

\newcommand{\marrow}{\marginpar{$\longleftarrow$}}

\newcommand{\maria}[1]{\textsc{\textcolor{red}{Maria says:}} \marrow \textsf{#1}}

\title{Induced equators in flag spheres}

\author{Maria Chudnovsky  and Eran Nevo}
\thanks{M. Chudnovsky was partially supported by NSF grants DMS-1763817 and by
  ERC advanced grant 320924.
  This material is based upon work supported in part by the U. S. Army
  Research Office under grant number  W911NF-16-1-0404.
E. Nevo was partially supported by the Israel Science Foundation grant ISF-1695/15, by ISF-BSF joint grant 2016288, and by ISF-NRF Singapore joint research program grant 2528/16.}

\begin{document}
\maketitle
\begin{abstract}
We propose a combinatorial approach to the following strengthening of Gal's conjecture:  $\gamma(\Delta)\ge \gamma(E)$ coefficientwise, where $\Delta$ is a flag homology sphere and $E\subseteq \Delta$ an induced homology sphere of codimension $1$. We provide partial evidence in favor of this approach, and prove a nontrivial nonlinear inequality that follows from the above conjecture, for boundary complexes of flag $d$-polytopes: $h_1(\Delta) h_i(\Delta) \ge (d-i+1)h_{i-1}(\Delta) + (i+1) h_{i+1}(\Delta)$ for all $0\le i\le d$.
\end{abstract}

\section{Introduction}
The $\gamma$-vector, reviewed in the next section, encodes the face numbers of simplicial complexes which are homology spheres.
These complexes are \emph{flag} if equal the clique complex of their $1$-skeleton, for example barycentric subdivisions of the boundary complex of polytopes. Gal~\cite{Gal} conjectured the following tight analog of the  GLBT inequalities~\cite{McMullen-g-conj, Stanley:gThm, Adiprasito-g} in the flag case:

\begin{conjecture}[Gal~\cite{Gal}]\label{conj:Gal}
If $\Delta$ is the boundary complex of a flag polytope, or more generally a flag homology sphere, then
$\gamma(\Delta)\geq 0$, coefficientwise.
\end{conjecture}
This conjecture includes the Charney--Davis' conjecture~\cite{Charney-Davis} as a special case.

As a vertex link in a flag homology sphere is again a flag homology sphere, the following conjecture immediately implies Gal's conjecture.

\begin{conjecture}[Link Conjecture]\label{conj:link}
If $v$ is a vertex in a flag homology sphere $\Delta$, then its link satisfies
$\gamma(\lk_v(\Delta))\leq \gamma(\Delta)$, coefficientwise.
\end{conjecture}
An \emph{equator} in a flag homology sphere is any \emph{induced} subcomplex which is a flag homology sphere of codimension $1$, see e.g.~\cite{Labbe-Nevo}. Each vertex link is an example of an equator. We prove in Proposition~\ref{prop:link->E} that the following formal generalization of Conjecture~\ref{conj:link} is in fact equivalent to it.

\begin{conjecture}[Equator Conjecture]\label{conj:E}
If $E$ is an equator in a flag homology sphere $\Delta$, then
$\gamma(E)\leq \gamma(\Delta)$, coefficientwise.
\end{conjecture}

Athanasiadis~\cite{Athanasiadis:FlagSd} showed that Gal's conjecture follows from his conjecture that $\gamma(\Delta')\leq \gamma(\Delta)$ when $\Delta$ is a certain subdivision of $\Delta'$, called vertex-induced homology subdivision.
However, the Link Conjecture, which amounts to dimension reduction, does not follow from Athanasiadis' conjecture, as $\Delta$ may not be such subdivision of the suspension of $\lk_v(\Delta)$. To see this, let $\Delta$ be the boundary of the icosahedron. Then $\lk_v(\Delta)$ is the $5$-cycle for every vertex $v$,
and therefore every vertex of $\Delta$ has degree five, and in particular there are three pairwise-adjacent vertices of degree five. However, in the
suspension $\Delta'=\Sigma_{a,b}\lk_v(\Delta)$ there are only two vertices of degree five, namely $a$ and $b$. Since the inverse image of every vertex $w$  of $\Delta'$ in $\Delta$  has at least the same degree as $w$,
it follows that two adjacent vertices of $\Delta$ have $a$ (or $b$) as their
image, which is impossible.


Let $\mathcal{R}$ be the subfamily of \emph{minimal} flag homology spheres, i.e. those that do not admit edge contractions that keep them flag, excluding the octahedral ones;
equivalently, those where each edge belongs to an induced $4$-cycle, excluding the octahedral sphere in each dimension.
It is known and easy that Gal's conjecture reduces to proving it for all $\Delta\in \mathcal{R}$ (see Lemma~\ref{lem:gamma-basicConstructions}), and in~\cite[Conj.6.1]{Lutz-Nevo} it is conjectured that $\gamma_2(\Delta)>0$ for all $\Delta\in
\mathcal{R}$.
In Proposition~\ref{prop:equator-edge_contract} we show that the Equator Conjecture holds if it holds for all $\Delta\in
\mathcal{R}$.
Unconditionally, we verify the validity of the Link conjecture for the following family: Let
$\mathcal{S}$ be the family of boundary complexes of flag polytopes obtained from a crosspolytope by successive edge subdivisions.

\begin{proposition}\label{prop:EforAisbett}
Conjecture~\ref{conj:link}
holds for all $\Delta\in \mathcal{S}$.
\end{proposition}
Replacing Conjecture~\ref{conj:link} with~\ref{conj:E} in above proposition is left open.
We remark that Aisbett~\cite{Aisbett-sd} and Volodin~\cite{Volodin:FFK} proved that for any $\Delta\in \mathcal{S}$, $\gamma(\Delta)$ is the $f$-vector of some flag complex, supporting a conjecture of Nevo and Petersen~\cite{Nevo-Petersen}.

We show in Proposition~\ref{prop:tentative->E} that Conjecture~\ref{conj:E} follows from the following structural conjecture.

\begin{problem}[Structure]\label{prob:tentative}
For all homology flag spheres $\Delta$, one of the following three alternatives must hold:

(0) $\Delta$ is a suspension, or

(i) there exists an edge in $\Delta$ which belongs to no induces $4$-cycle, or

(ii) for every vertex $v\in \Delta$ there exists an equator $E$ in $\Delta$ which is not a vertex link and which does not contain $v$.
\end{problem}
Observe that (0) or (i) must hold if some vertex $v$ in $\Delta$ is nonadjacent to \emph{at most two} vertices: never to zero as $\Delta$ is not a cone, if to exactly one then $\Delta$ is a suspension (over $v$ and the unique nonneighbor of it), and if to exactly two then the two nonneighbors of $v$ form an edge which is in no induced $4$-cycle by~\cite[Lem.3.4]{Labbe-Nevo}.
We prove the structural conjecture above holds when the dimension of $\Delta$ is at most two in Theorem~\ref{thm:dim2}, using this observation.

Let $\Delta_0$ denote the vertex set of $\Delta$. Note that the Link Conjecture implies the average assertion $\sum_{v\in \Delta_0}\gamma(\lk_v(\Delta))\le f_0(\Delta)\gamma(\Delta)$, which implies the $h$-polynomial inequality
\begin{equation}\label{eq:h-ineq}
(1+t)\sum_{v\in \Delta_0} h_{\lk_v(\Delta)}(t)\le f_0(\Delta) h_{\Delta}(t).
\end{equation}
Recall that McMullen's proof of the UBT for polytopes used the inequality
\[\sum_{v\in \Delta_0} h_{\lk_v(\Delta)}(t)\le f_0(\Delta) h_{\Delta}(t).\]
The inequality (\ref{eq:h-ineq}) gives stronger upper bounds for flag homology spheres, however they are not tight. See~\cite{Nevo-Petersen, Adamaszek-Hladky:flagUBC, Zheng-flag_UBT5} for the statement and progress on the UBC for flag homology spheres. Here we prove (\ref{eq:h-ineq}) in the polytope case:

\begin{theorem}\label{thm:h-ineq}
The inequality (\ref{eq:h-ineq}) holds for all flag polytopes; it is tight only for the crosspolytopes.
\end{theorem}
The proof combines a simple shelling argument with the following result, which may be of independent interest. A {\em half-integral perfect matching } in a graph $G$ is a function $f:E(G) \rightarrow \{0,1, \frac{1}{2}\}$ such that
for every vertex $v$ of $G$, $\Sigma_{e \text{ incident with }v}f(e)=1$.
Given a graph $H$, a graph $G$ is {\em the complement of $H$} if $G$ has the same vertex set as $H$, and two vertices are adjacent in $G$ if and only if they are
non-adjacent in $H$. We prove:

\begin{theorem}\label{thm:matching}
Let $G$ be the complement of the $1$-skeleton of a flag homology sphere. Then $G$ has a half-integral perfect matching; equivalently, the vertex set can be partitioned into a matching and odd cycles in $G$.
\end{theorem}

Combining Theorem~\ref{thm:h-ineq} with McMullen's formula (see e.g.\cite[Prop.2.3]{Swartz:k-CM})
\[\sum_{v\in \Delta_0} h_{i-1}(\lk_v(\Delta))=ih_i(\Delta)+(d-i+1)h_{i-1}(\Delta)
\]
gives the following inequality on the $h$-vector, which seems new:
\begin{corollary}\label{cor:h-bound}
For $\Delta$ the boundary complex of a flag $d$-polytope, its $h$-vector satisfies
\begin{equation}\label{eq:h1hi-ineq}
h_1 h_i \ge (d-i+1)h_{i-1} + (i+1) h_{i+1}
\end{equation}
for all $i$.
\end{corollary}
For comparison, for the cyclic $d$-polytope and $i<d/2$,
$h_1 h_i < i h_{i+1}$.
In Section~\ref{sec:Balanced} we show that (\ref{eq:h1hi-ineq}) holds for boundary complexes of balanced $d$-polytopes as well, namely when the $1$-skeleton is vertex $d$-colorable.

\textbf{Outline.} In Sec.~\ref{sec:Prelim} we set notation, recall the $\gamma$-vector and its relation to vertex splits and other basic constructions.
In Sec.~\ref{sec:E} we prove results towards the Equator Conjecture.
In Sec.~\ref{sec:half} we prove Theorems~\ref{thm:h-ineq} and~\ref{thm:matching}.
Sec.~\ref{sec:Balanced} applies ideas from Sec.~\ref{sec:half} to the balanced case.

\section{Preliminaries}\label{sec:Prelim}
For the basics on face enumeration needed here we refer to e.g. Stanley's book~\cite{Stanley:CombinatoricsCommutativeAlgebra-96}
or the recent surveys by Klee-Novik~\cite{Klee-Novik:f-survey} and Zheng~\cite{Zheng:flag-survey}; for basics on polytopes refer to e.g. the textbooks by Gr\"{u}nbaum~\cite{Grunbaum:ConvexPolytopes-03} and Ziegler~\cite{Ziegler}.

\subsection{Simplicial complexes}
A \emph{simplicial complex} $\Delta$ is a finite collection of subsets of $[n]:=\{1,\dots, n\}$, called \emph{faces},  closed under containment.
A face of cardinality $k+1$ has \emph{dimension}~$k$, called a $k$-face; the dimension of $\Delta$ is $\dim\Delta:=-1+\max\{|\sigma|:\ \sigma\in\Delta\}$.
Faces of dimension $0$ (resp. $1$) are called \emph{vertices} (resp. \emph{edges}); together they form the $1$-skeleton, or \emph{graph} of $\Delta$.
Then $\Delta$ is \emph{flag} if its faces are exactly the cliques over its graph.

Given a face $\sigma\in\Delta$, the (closed) ~\emph{star},~\emph{antistar}, and~\emph{link} of~$\sigma$ in~$\Delta$ are the following subcomplexes of $\Delta$:
\begin{align*}
\st_{\sigma}\Delta & := \{\tau \in \Delta : \sigma\cup \tau \in \Delta \}, \\
\antist_{\sigma}\Delta & := \{\tau\in \Delta : \sigma\not\subseteq \tau\}, \\
\lk_{\sigma}\Delta & := \{\tau \in \Delta : \sigma \cup \tau \in \Delta,~\ \sigma \cap \tau=\emptyset\}.
\end{align*}
Then for any vertex $v\in\Delta$, $\Delta=\st_{v}\Delta \cup_{\lk_{v}\Delta}\antist_{v}\Delta$.
(We will keep abusing notation writing $v$ for the singleton $\{v\}$.)

Call $\Delta$ a \emph{homology sphere} (over a field $\mathbb{F}$) if for all faces $\sigma\in \Delta$ the reduced homology groups with coefficients in $\mathbb{F}$ satisfy
\[ \widetilde{H}_i(\lk_\sigma\Delta,\mathbb{F})=\begin{cases} 0 & \text{if } i<\dim\Delta-|\sigma|, \\ \mathbb{F} & \text{if } i=\dim\Delta-|\sigma|.\end{cases}
\]
Call $\Delta$ pure if all its maximal faces (w.r.t. inclusion) have the same dimension.
A pure
$(d-1)$-dimensional simplicial complex $B$ is a \emph{homology ball} (over $\mathbb{F}$) if (i) for all faces $\sigma\in \Delta$ the link $\lk_\sigma\Delta$ is either a $(d-1-|\sigma|)$-dimensional homology sphere over $\mathbb{F}$ or is homologically $\mathbb{F}$-acyclic, and (ii) the faces of $B$ with acyclic link form a
$(d-2)$-dimensional homology sphere over $\mathbb{F}$.

Recall a subcomplex $X$ of $\Delta$ is \emph{induced} if $X=\Delta[W]:=\{\sigma\in\Delta:\ \sigma\subseteq W\}$ for some subset $W$ of the vertex set of $\Delta$.
Given a homology sphere $\Delta$, an induced codimension 1 homology sphere $E\subseteq\Delta$ is called an \emph{equator} of~$\Delta$.
By Jordan--Alexander theorem, $\Delta$ is decomposed into two homology balls intersecting in $E$, denoted $\Delta=B_1 \cup_E B_2$.

The \emph{join} of two simplicial complexes $\Delta_i$, $i=1,2$, on disjoint vertex sets is
\[\Delta_1*\Delta_2:=\{\sigma_1\cup\sigma_2:~\ \sigma_i\in\Delta_i,\ i=1,2\}.\]
Important instances are the case of a \emph{cone}, where $\Delta_2=\{\emptyset,\{v\}\}$ and we simply write $\Delta_1*v$ for the join, and the case of \emph{suspension},
where $\Delta_2=\{\emptyset,\{v\},\{u\}\}$ and we simply write $\susp_{u,v}\Delta_1$ for the join.
The join of the two-point complex with itself $d$ times is the \emph{octahedral} $(d-1)$-sphere; it can be realized as the boundary of the $d$-crosspolytope.
It is the unique minimizer of the number of vertices (and $i$-faces, for all $i$) among all flag homology $(d-1)$-spheres, e.g.~\cite{Gal}, \cite{Meshulam-domination}.

The \emph{contraction} of $\Delta$ by an edge $e=uv\in\Delta$ is the complex
\[\Delta':=\antist_v\Delta\cup u* \antist_u\st_v\Delta,\]
obtained by replacing $v$ by $u$ in faces containing $v$ in $\Delta$. Then $\Delta$ is obtained from $\Delta'$ by a \emph{vertex split} at $u$.
We recall the following known facts, see e.g.~\cite[Lem.2.1]{Labbe-Nevo}.

\begin{lemma}\label{lem:elementary}
Let $\Delta$ be a $(d-1)$-dimensional flag homology sphere, $\sigma\in\Delta$, and $e\in \Delta$ an edge. Then:
\begin{enumerate}[label=\roman{enumi}),ref=\ref{lem:elementary}~\roman{enumi})]
 \item The link $\lk_\sigma\Delta$ is a flag induced homology sphere, hence an equator when $\sigma=\{v\}$. \label{lem:Llink_flag}
 \item The contraction of $\Delta$ by $e$ is a flag homology sphere if and only if $e$ is not contained in an induced $4$-cycle in the graph of $\Delta$. \label{lem:Lcontract_flag}
\end{enumerate}
\end{lemma}
A particularly simple case of vertex split is that of stellar subdivision of $\Delta'$ at an edge $e=uv$, by introducing say a new vertex $v_e$. This operation preserves being flag. Then the inverse operation is contracting the edge $uv_e$ (or $vv_e$, they both give back the original complex).
In the case when $\Delta'=\bd P$ is the boundary complex of a simplicial polytope $P$, subdividing $e$ can be realized by placing $v_e$ \emph{beyond} $e$, thus the resulted $\Delta$ is again the boundary complex of a simplicial polytope.


\subsection{$f-,h-,\gamma$-vectors}
For a $(d-1)$-dimensional simplicial complex $\Delta$ let $f_i(\Delta):=|\{\sigma\in\Delta:\ |\sigma|=i+1\}|$ denote the number of $i$-faces of $\Delta$, and $f(\Delta):=(f_{i-1}(\Delta))_{i=0}^{d}$ denote its $f$-vector; equivalently, let $f_{\Delta}(t):=\sum_{i=0}^{d} f_{i-1}(\Delta)t^i$ denote its $f$-polynomial.

Define the $h$-polynomial and $h$-vector of $\Delta$ by the equality
\[x^d \sum_{i=0}^d h_i(\Delta)(\frac{1}{x})^i =
(x-1)^d \sum_{i=0}^d f_{i-1}(\Delta)(\frac{1}{x-1})^i.
\]
When $\Delta$ is a flag homology sphere,
the Dehn-Sommerville relations assert that $h_i(\Delta)=h_{d-i}(\Delta)$ for all $0\le i\le d$. Being a palindrome, one can express $h_{\Delta}(t)$ as
\[h_{\Delta}(t)=\sum_{i=0}^{\lfloor\frac{d}{2}\rfloor} \gamma_i t^i (1+t)^{d-2i},\]
and the $\gamma_i$s define the $\gamma$-vector and $\gamma$-polynomial of $\Delta$, namely $\gamma_{\Delta}(t)=\sum_{i=0}^{\lfloor\frac{d}{2}\rfloor} \gamma_i t^i$.
(We will switch between the $f,h,\gamma$-vectors and polynomials freely as convenient, where coefficientwise $\ge$ or $=$ between vectors of different length means by interpreting them as polynomials.)

We collect the following easy facts on the behavior of $\gamma$-polynomials
under basic constructions.
\begin{lemma}[\cite{Gal}]\label{lem:gamma-basicConstructions}
Let $\Delta$ be a flag homology sphere and $e\in\Delta$ an edge.
Then,
\begin{enumerate}[label=\roman{enumi}),
ref=\ref{lem:gamma-basicConstructions}~\roman{enumi})]
 \item
 the suspension satisfies $\gamma(\susp_{a,b}\Delta)=\gamma(\Delta)$;
\label{lem:gamma-susp}
\item
for $\Delta'$ the contraction of $\Delta$ by $e$, $\gamma_{\Delta}(t)=\gamma_{\Delta'}(t)+
t\gamma_{\lk_e\Delta}(t)$.
\label{eq:gamma-contraction}
\end{enumerate}
\end{lemma}


\section{Towards the Equator conjecture}\label{sec:E}

First, we reduce the Equator conjecture to the Link conjecture.
\begin{proposition}~\label{prop:link->E}
Fix $d$ and $n$. Then the assertion of
Conjecture~\ref{conj:link} holds for all homology $d$-spheres with at most $n$ vertices if and only if the assertion of
Conjecture~\ref{conj:E} holds for all homology $d$-spheres with at most $n$ vertices.
\end{proposition}


\begin{proof}
Let $\Delta$ be a flag homology sphere.
As every vertex link is an induced subcomplex, and hence an equator, the assertion of Conjecture~\ref{conj:E} for $\Delta$ clearly implies the assertions of
Conjecture~\ref{conj:link} for $\Delta$.

For the converse implication,
let $E$ be an equator of $\Delta$ and not a vertex link. Thus, it decomposes $\Delta$ as the union of two homology balls $B_1$ and $B_2$ with common boundary $E$, such that in each $B_i$ there are at least \emph{two} interior vertices.

Consider the flag homology spheres $\Delta_i=B_i\cup (E*v_i)$ where the cone vertex $v_i$ of $E$ is not in $B_i$, for $i=1,2$.
Then the $f$-polynomials satisfy \[f_{\Delta}(t)=f_{\Delta_1}(t)+f_{\Delta_2}(t)-f_{\Sigma_{v_1,v_2}E}(t).\]
Translating into $\gamma$-polynomials, and using the fact that suspension does not change the $\gamma$-polynomial, gives
\begin{equation}\label{eq:gammaD1D2decomposition}
\gamma_{\Delta}(t)=\gamma_{\Delta_1}(t)+\gamma_{\Delta_2}(t)-\gamma_{E}(t).
\end{equation}
Now,
$E$ is a vertex link in $\Delta_i$,
and $\Delta_i$ has fewer vertices then $\Delta$ and same dimension as $\Delta$,
so by Conjecture~\ref{conj:link} $\gamma_E(t)\le \gamma_{\Delta_i}(t)$. Combining with (\ref{eq:gammaD1D2decomposition}) gives $\gamma_{\Delta}(t)\ge 2\gamma_E(t)-\gamma_E(t)=\gamma_E(t)$ as claimed.
\end{proof}

Next, we reduce the Equator conjecture for all flag homology spheres to the subfamily of minimal ones.
\begin{proposition}~\label{prop:equator-edge_contract}
Let $e=uv$ be an edge in a flag homology sphere $\Delta$ and in no induced $4$-cycle in $\Delta$.
If Conjecture~\ref{conj:E} holds for all flag homology spheres of dimension $\le \dim\Delta$ and with $<f_0(\Delta)$ vertices (and all equators in them), then it holds for $\Delta$ (and all equators $E$ in $\Delta$).
\end{proposition}
\begin{proof}
Let $E$ be an equator in $\Delta$.
There are exactly 3 cases: either (i) $e$ is disjoint from $E$, or (ii) $e$ is contained in $E$, or (iii) $e$ intersects $E$ in a single vertex, say $u$.

As $e=uv$ is in no induced $C_4$, its contraction results in a smaller flag homology sphere $\Delta'$. In case (i) we get by induction $\gamma(E)\le \gamma(\Delta')$, so done as $\gamma(\Delta')\le \gamma(\Delta)$ by Lemma~\ref{eq:gamma-contraction}.
In case (ii) the contraction of $e$ in $E$ results in an equator $E'$ of $\Delta'$.
Note that $\lk_eE$ is an equator of $\lk_e\Delta$,
so by induction and Lemma~\ref{eq:gamma-contraction} we get
\[\gamma_{\Delta}(t)=\gamma_{\Delta'}(t)+t\gamma_{\lk_e\Delta}(t)
\ge \gamma_{E'}(t)+t\gamma_{\lk_eE}(t)=\gamma_{E}(t)
.\]
In case (iii),
as
$E$ decomposes $\Delta$ into the union of two homology balls, $\Delta=B_1\cup_E B_2$, and the complexes $\Delta_i=B_i\cup(E *v_i)$ and $\Delta$ satisfy equation (\ref{eq:gammaD1D2decomposition}).
If $f_0(\Delta)>f_0(\Delta_i)$ for $i=1,2$ then by induction $\gamma(E)\le \gamma(\Delta_i)$
for $i=1,2$ and combined with (\ref{eq:gammaD1D2decomposition}) we are done.

Else, w.l.o.g. $e$ is contained in $B_1$.

\textbf{Case} $f_0(\Delta)=f_0(\Delta_1)$: $e$ is in no induced $C_4$ in $\Delta_1$, hence $\lk_v\Delta\cap E\subseteq \st_uE$, so contracting $e$ results in $\Delta'$ where $E$ remains an equator (namely, remains induced).
By induction and Lemma~\ref{eq:gamma-contraction} conclude
$\gamma(E)\le \gamma(\Delta')\le \gamma(\Delta)$.

\textbf{Case} $f_0(\Delta)=f_0(\Delta_2)$: then $E=\lk_v\Delta$.

As $uv$ is in no induced $C_4$, the boundary of the homology ball $B=\st_v\Delta\cup \st_u\Delta$ is an induced subcomplex of $\Delta$, denote it by $E''$; so $E''$ is an equator of $\Delta$.
Consider the flag homology sphere $\Delta''=B\cup E''*w$ with $w$ not a vertex of $B$.
Applying (\ref{eq:gammaD1D2decomposition}) to $\Delta$, $\Delta''$ and the third sphere $\Delta'''$ obtained by coning the boundary of the complementary ball to $B$, we get by induction $\gamma(E'')\le \gamma(\Delta''')$; hence $\gamma(\Delta'')\le \gamma(\Delta)$.

\textbf{If} $\Delta\neq \Delta''$ then
by induction $\gamma(E)\le\gamma(\Delta'')$
and we are done.

\textbf{Else}, $\Delta=\Delta''$.
Note that in this case $\Delta$ is the union of the stars of $u,v$ and $w$; we have
\[f_{\Delta}(t)=(1+t)
(f_{\lk_v\Delta}(t)+f_{\lk_u\Delta}(t))
-(1+t)^2f_{\lk_{e}\Delta}(t) +tf_{\lk_w\Delta}(t)
.\]
Translating into $h$-polynomials we get
\[h_{\Delta}(t)=
h_{\lk_v\Delta}(t)+h_{\lk_u\Delta}(t)
-h_{\lk_{e}\Delta}(t) +th_{\lk_w\Delta}(t).\]
Further, contracting $e$ in $\Delta$ gives the suspension over $\lk_w\Delta$ which by Lemma~\ref{eq:gamma-contraction} gives
\[h_{\Delta}(t)=
 th_{\lk_{e}\Delta}(t) +(1+t)h_{\lk_w\Delta}(t).\]
Equating the RHSs of the last two equations gives in $\gamma$-terms
\[\gamma(\lk_w\Delta)+\gamma(\lk_e\Delta)=
\gamma(\lk_v\Delta)+\gamma(\lk_u\Delta).\]
By Lemma~\ref{eq:gamma-contraction} and induction $\gamma(\lk_w\Delta)\le \gamma(\Delta)$,
 and by induction
$\gamma(\lk_e\Delta)\le \gamma(\lk_u\Delta)$, thus
\[\gamma(\Delta)\ge \gamma(\lk_w\Delta)=\\
\gamma(\lk_v\Delta)+(\gamma(\lk_u\Delta)-\gamma(\lk_e\Delta))
\ge\\
 \gamma(\lk_v\Delta)=\gamma(E)
,\]
completing the proof.
\end{proof}

Recall the family $\mathcal{R}$ from the Introduction, of minimal flag homology spheres.  Proposition~\ref{prop:equator-edge_contract} immediately implies:
\begin{corollary}
If Conjecture~\ref{conj:E} holds for all $\Delta\in \mathcal{R}$  then it holds in general.
\end{corollary}

Next we discuss Proposition~\ref{prop:EforAisbett}.
In order to prove it 
we need the following straightforward observation:

\begin{lemma}\label{lem:LinkForAisbett}
The family $\mathcal{S}$ is closed under
(i) suspension and
(ii) links.
\end{lemma}
\begin{proof}
For (i), note that if $\Delta$ is obtained from a homology sphere $\Delta'$ by stellar subdivision at the edge $e\in \Delta'$, then the suspension $\Sigma_{a,b}\Delta$ is obtained from $\Sigma_{a,b}\Delta'$ by stellar subdivision at the same edge $e$.

For (ii), the assertion clearly holds for octahedral spheres. We argue by induction.
Keeping the notation of the proof of part (i), we distinguish cases according to the vertex $v$ whose link is being considered, for $v\in \Delta\in \mathcal{S}$: cases are (1.)$v\in e$, (2.)$v=v_e$ is the new vertex, (3.)$v\in\lk_e\Delta'$, and (4.)otherwise. See e.g.~\cite[Sec.3]{Aisbett-sd} for details. Specifically,
in case (1.) $\lk_v\Delta\cong\lk_v\Delta'$ and we are done by induction on number of vertices, in case (2.) $\lk_{v_e}\Delta\cong\susp_{a,b}\lk_e\Delta'$ so
we are done using part (i), in case (3.) $\lk_v\Delta$ is obtained from $\lk_v\Delta'$ by a stellar subdivision at the edge $e$ so we are done by induction on dimension, and in case (4.) $\lk_v\Delta=\lk_v\Delta'$ and there is nothing new to prove.
\end{proof}

\begin{proof}[Proof of Proposition~\ref{prop:EforAisbett}]
Let $v\in \Delta\in \mathcal{S}$, and $\Delta$ obtained from $\Delta'$ by a stellar subdivision at edge $e$. Consider the 4 cases in the proof of Lemma~\ref{lem:LinkForAisbett}, whose assertion we also use.

In case (1.), \[\gamma(\lk_v\Delta)=\gamma(\lk_v\Delta')\le \gamma(\Delta')\le \gamma(\Delta),\]
where first inequality is by induction and second one is by Lemma~\ref{eq:gamma-contraction}, where nonnegativity of $\gamma(\lk_e(\Delta'))$ is known by Lemma~\ref{lem:LinkForAisbett} and induction.

In case (2.),
\[\gamma(\lk_v\Delta)=\gamma(\lk_e\Delta')\le
\gamma(\Delta')\le \gamma(\Delta),\]
where for the first inequality we applied induction twice, as for $e=uw$, $\lk_e\Delta'=\lk_u(\lk_v\Delta')$.

In case (3.), \[\gamma(\lk_v\Delta)=\gamma(\lk_v\Delta')+t\gamma(\lk_e(\lk_v\Delta'))\le
\gamma(\Delta')+t\gamma(\lk_e\Delta')=\gamma(\Delta),\]
where we used that for a partition of a face $\sigma=\sigma_1\cup\sigma_2$, links operators satisfy $\lk_{\sigma}\Delta'=\lk_{\sigma_2}(\lk_{\sigma_1}\Delta')=
\lk_{\sigma_1}(\lk_{\sigma_2}\Delta')$.

In case (4.) we are immediately done by induction.
\end{proof}

Next we consider relevance of the Structure conjecture in Problem~\ref{prob:tentative}.

\begin{proposition}~\label{prop:tentative->E}
Problem~\ref{prob:tentative} implies Conjecture~\ref{conj:E}.
\end{proposition}
\begin{proof}
  Let $\Delta$ be a flag homology sphere.
  By Proposition~\ref{prop:link->E}, it is enough to show that
  for every vertex $w$ of $\Delta$ we have
$\gamma(\lk_w(\Delta)) \leq \gamma (\Delta)$.
    We may assume that the assertion of Conjecture~\ref{conj:E} holds for all
  flag homology spheres $\Delta'$ such that $f_0(\Delta')<f_0(\Delta)$ and $\dim(\Delta')\leq \dim(\Delta)$, and that
  one of the outcomes listed in Problem~\ref{prob:tentative} holds for
  $\Delta$. Thus either

(0) $\Delta$ is a suspension, or

(i) there exists an edge in $\Delta$ which belongs to no induced $4$-cycle, or

(ii) for every vertex $v\in \Delta$ there exists an equator $E$ in $\Delta$ which is not a vertex link and which does not contain $v$.

  Let $w$ be a vertex of $\Delta$.
  Assume first that $\Delta=\Sigma_{a,b}\Delta'$.
  If $w \in \{a,b\}$, then
    the result follows immediately from the first statement of
    Lemma~\ref{lem:gamma-basicConstructions}.  Thus $w \in \Delta'$, and
    $\lk_w(\Delta)=\Sigma_{a,b}\lk_w(\Delta')$. 
    Inductively
  we have that $\gamma(\lk_w(\Delta')) \leq \gamma (\Delta')$. But now, again by the first statement
  of Lemma~\ref{lem:gamma-basicConstructions}, we deduce:
  $$\gamma(lk_w(\Delta))=\gamma(\lk_w(\Delta')) \leq \gamma(\Delta')=\gamma(\Delta)$$
  and thus the assertion of Conjecture~\ref{conj:E} holds for
  $\Delta$.

  If there exists an edge in $\Delta$ which belongs to no induced $4$-cycle,
  then the assertion of Conjecture~\ref{conj:E}
  for   $\Delta$ follows immediately from
  Proposition~\ref{prop:equator-edge_contract}.

Thus we may assume that
outcome (ii) above holds.
Let $E$ be an equator in $\Delta$ which is not a vertex link and which does not contain $w$ (such $E$ exists since outcome (ii) holds).  Then $E$  decomposes $\Delta$ as the union of two homology balls $B_1$ and $B_2$ with common boundary $E$, such that in each $B_i$ there are at least \emph{two} interior vertices. We may assume that $w$ is an interior vertex of $B_1$, and
therefore $\lk_w(\Delta)$ is contained in $B_1$.
Consider the flag homology spheres $\Delta_i=B_i\cup (E*v_i)$ where the cone vertex $v_i$ of $E$ is not in $B_i$, for $i=1,2$.
Since $E$ is not a vertex link, we deduce that for $i=1,2$
$f_0(\Delta_i) < f_0(\Delta)$. Consequently,
$\gamma(\lk_w(\Delta_1)) \leq \gamma(\Delta_1)$ and
$\gamma(E) \leq \gamma(\Delta_2)$.
 Note that  $\lk_w(\Delta)=\lk_w(\Delta_1)$.
 We now use  \eqref{eq:gammaD1D2decomposition} to deduce the following
 (coefficientwise):
$$ \gamma(\Delta)=\gamma(\Delta_1)+\gamma(\Delta_2) -\gamma(E) \geq \gamma(\Delta_1) \geq \gamma(\lk_w(\Delta_1)) = \gamma(\lk_w(\Delta)),$$
as required.
\end{proof}

The structure conjectured in Problem~\ref{prob:tentative} clearly holds for spheres of dimension $\le 1$.
Further, it holds in dimension $2$ due to:
\begin{lemma}\label{lem:flagWhiteley}
If $\Delta$ is a flag (homology) $2$-sphere, different from the octahedron's boundary, then there exists an edge $e\in \Delta$ such that $e$ is not contained in any induced $4$-cycle.
\end{lemma}
This statement is a flag analog of Whiteley~\cite[Lem.6]{Whiteley-VertexSplitting}; we omit its simple proof.
Thus, for flag $2$-spheres one of the alternatives (0) and (i) in Problem~\ref{prob:tentative} holds.
The point in Theorem~\ref{thm:dim2} below is to show how alternative (ii) in Problem~\ref{prob:tentative} can be found, when a strong condition that implies (0) or (i) fails to hold.



\begin{theorem}~\label{thm:dim2}
For every vertex $v$ in a flag (homology) $2$-sphere $\Delta$, either

(i) some vertex of $\Delta$ is non-adjacent to at most two vertices, or

(ii) there exists an equator $E$ in $\Delta$ which is not a vertex link and does not contain $v$.
\end{theorem}

\begin{proof}

  Let $v$ be a vertex of $\Delta$. Since $\Delta$ is a flag homology
  $2$-sphere,  $\lk_v(\Delta)$ is an induced cycle. Let the vertices of
 $\lk_v(\Delta)$ be $u_1, \ldots, u_i$, where $u_iu_{i+1}$ is an edge of
$\Delta$ for   every $i \in \{1, \ldots, t-1\}$, and $u_1$ is adjacent to $u_t$.
Then  there are no other adjacent pairs among $\{u_1, \ldots, u_t\}$.
If no vertex of $\lk_{u_t}(\Delta) \setminus \st_v(\Delta)$ has a neighbor in
$\lk_v(\Delta) \setminus \st_{u_t}(\Delta)$, then the edge $vu_t$ is in no
induced  $C_4$, and so $E=(\lk_v(\Delta) \cup \lk_{u_t}(\Delta))\setminus\{v,u_t\}$ is an induced cycle, and therefore an equator. If there exists $w$ such that
  $E=\lk_w(\Delta)$, then outcome (i) holds, and otherwise outcome (ii) holds.
  Thus by symmetry we may assume that for every $i \in \{1, \ldots, t\}$
$u_i$  has a neighbor $w_i \not \in \st_v(\Delta)$ and such that
$w_i$ has a neighbor in $\lk_v(\Delta) \setminus \st_{u_i} (\Delta)$.

For
a vertex $w$,
a {\em $w$-interval} is a circular interval
$[u_j, \ldots, u_k]$ of  $\lk_v(\Delta)$ such that $u_j$ is non-adjacent to $u_k$, $w$ is adjacent to $u_j, u_k$, and $w$ has no other neighbor in
this interval.

Now, no $w_i$-interval exists iff $\lk_{w_i}(\Delta)=\lk_{v}(\Delta)$ in which case outcome (i) holds. Thus, we may assume there exists a $w_i$-interval $I_{w_i}$ for all $1\le i\le t$.
Let $C_{w_i}$ be the induced graph in $\Delta$ on the vertex set $I_{w_i}\cup\{w_i\}$. Then $C_{w_i}$ is an induced cycle not containing $v$.
If $C_{w_i}$ is not a vertex link then outcome (ii) holds; thus assume $C_{w_i}=\lk_{s_i}(\Delta)$ for some vertex $s_i$, for all $1\le i\le t$.

This implies a specific structure on $\Delta$, which means both outcomes (i) and (ii) hold, as follows.
By renaming we may assume that, for some $i$, $I_{w_i}=[u_1,\ldots,u_k]$ is a $w_i$-interval.
Then $\lk_{s_i}(\Delta)\cap \lk_v(\Delta)=I_{w_i}$, thus the unique $s_i$-interval is $[u_k,\ldots,u_1]$. Regard now $s_i$ as $w_2$, then $s_2=w_i$, and $\Delta$ has exactly two vertices outside $\st_v(\Delta)$, and outcome (i) holds.
Also, outcome (ii) holds, as $(w_2,u_k,\ldots,u_1)$ is an equator.
\end{proof}

\section{Half-integral matchings}\label{sec:half}
Here we prove Theorems~\ref{thm:h-ineq} and~\ref{thm:matching}.
For background on shelling see e.g.~\cite{Ziegler}.

Let $u,v$ be two vertices of a simplicial $d$-polytope $P$ ($d\ge 2$) such that $uv$ is not an edge of $P$. Consider the line through $u$ and $v$, and perturb it to obtain an oriented line $l$ that crosses each facet hyperplane in a different point; and the line shelling it defines shells the facets containing $v$ first and the facets containing $u$ last.
By the expression for $h_{\partial P}(t)$ in terms of the shelling one has:

\begin{lemma}\label{lem:uv-shelling}
For all nonedges $uv$ as above,
\[h_{\lk_v\partial P}(t)+t h_{\lk_u\partial P}(t)\le h_{\partial P}(t)
.\]
Further, equality holds iff all facets of $P$ contain either $u$ or $v$, namely $\partial P$ is a suspension over the vertices $u$ and $v$.
\end{lemma}

\begin{proof}[Proof of Theorem~\ref{thm:h-ineq}]
By Theorem~\ref{thm:matching}, proved below, the vertex set of $P$ admits a partition into a matching and odd cycles in the complement of the $1$-skeleton of $P$.
Orient the edges in the odd cycles cyclically and consider each edge of the matching as a cyclically oriented $2$-cycle.

Summing the inequality of Lemma~\ref{lem:uv-shelling} over all oriented edges given above, gives (\ref{eq:h-ineq}).

For the equality case, again by Lemma~\ref{lem:uv-shelling}, it happens iff $\Delta$ is a suspension over \emph{each} nonedge. In particular the nonedges give a perfect matching, so $\Delta$ has the same graph as the $d$-crosspolytope, and by flagness we are done.
\end{proof}

Before we prove Theorem~\ref{thm:matching} we need the following lemma.
\begin{lemma} \label{lem:disjointfacets}
  Let $\Delta$ be a flag homology sphere and let $F$ be a facet of
  $\Delta$. Then there exists a facet $F'$ of $\Delta$ that is disjoint from
  $F$.
\end{lemma}

\begin{proof}
  The proof is by induction on the dimension of $\Delta$. Let $v \in F$. Then
  $F_1=F \setminus \{v\}$ is a facet of $\lk_v(\Delta)$.
    Inductively, there exists a facet $F_2$ in
    $\lk_v(\Delta)$ such that $F_2$ is disjoint from $F_1$. Since $\Delta$ is a
    flag homology sphere, each $F_i$
  is contained in two facets of $\Delta$, and therefore there exists a vertex
  $w \neq v$ of $\Delta$ such that $F_2 \cup \{w\}$ is a facet of $\Delta$.
  But now $F$ and $F'=F_2 \cup \{w\}$ are two disjoint facets of $\Delta$ as required.
\end{proof}

We will also use Theorem~2.2.4 of \cite{Matchings}:
\begin{lemma} \label{halfTutte}
   A graph $G$ has a  half-integral perfect matching if and only if for every $X \subseteq V(G)$, $G \setminus X$ has at most $|X|$ isolated vertices.
  \end{lemma}

\begin{proof}[Proof of Theorem~\ref{thm:matching}]

  Let $\Delta$ be a flag homology sphere, and let $G$ be the complement of the
  $1$-skeleton of $\Delta$.
  First we prove the $G$ has a half-integral perfect matching.
We need  to show that $G$ satisfies the assumption of Lemma~\ref{halfTutte}.
Let $X \subseteq V(G)$ and let $Y$ be the set of isolated vertices of $G\setminus X$. Then  $Y$ is a clique of the $1$-skeleton of $\Delta$,  and every vertex of $Y$ is adjacent to every vertex of $\Delta \setminus (X \cup Y)$. It follows that $Y$ is contained in a facet $F$ of $\Delta$. By Lemma~\ref{lem:disjointfacets} there is a
facet $F'$ of $\Delta$ disjoint from $F$. Since every vertex of $Y$ is adjacent to every vertex of $F' \setminus X$, we deduce that
$|F' \setminus X| + |Y|  \leq |F'|$. Consequently, $|F' \cap X| \geq |Y|$,
and so $|Y| \leq |X|$ as required.

The second assertion of  Theorem~\ref{thm:matching} now follows immdiatly
by Proposition~2.2.2 of \cite{Matchings}.
\end{proof}

\section{Balanced polytopes}\label{sec:Balanced}
In fact, (\ref{eq:h-ineq}) holds also for (completely) balanced simplicial polytopes, for a very similar reason as in the flag case, as we show in this section.
\begin{observation}\label{obs:balanced-NonEdge}
Let $\Delta$ be the boundary complex of a balanced $d$-polytope, and $v$ a vertex in $\Delta$. Then there exists another vertex $v\neq u\in \Delta$ such that $uv$ is not an edge in $\Delta$.
\end{observation}
(Just take $u$ of same color as $v$; it exists else $\Delta$ would be a cone over $v$, a contradiction.)
Using line shellings, starting with all facets containing $v$ and ending with all facets containing its non-neighbor $u$ as above, (\ref{eq:h-ineq}) follows from showing that the graph $G$ complementary to the graph of $\Delta$ admits a half-integral perfect matching; equivalently,
by showing that for any subset $X$ of the vertex set $\Delta_0$ of $\Delta$, there are at most $|X|$ isolated vertices in the induced graph $G[\Delta_0\setminus X]$.

Indeed, let $Y=\{y_1,\ldots,y_t\}$ be a maximal set of isolated vertices in $G[\Delta_0\setminus X]$. Then all vertices in $\Delta_0\setminus X$ are in the intersection of the closed stars $\st_{y_i}(\Delta)$. In particular, the induced graph on $Y$ in $\Delta$ is complete so they all have distinct colors. By  Observation~\ref{obs:balanced-NonEdge} there exist distinct $x_1,\ldots,x_t$ with $x_i$ of same color as $y_i$, $x_i\neq y_i$, and so $\{x_1,\ldots,x_t\}\subseteq X$, showing $|Y|\le |X|$.

Thus, Cor.~\ref{cor:h-bound} holds also when replacing \emph{flag} by \emph{balanced}. $\square$

\par
\textbf{Acknowledgements.}
We thank Hailun Zheng for helpful comments on an earlier version of this paper.
\bibliographystyle{plain}
\bibliography{gbiblio}
\end{document}